\documentclass[a4,12pt]{amsart}
\usepackage{amssymb,amstext,amsmath,amscd,amsthm,amsfonts,enumerate,latexsym}
\usepackage{color}
\usepackage[dvipdfmx]{graphicx}
\usepackage[all]{xy}
\theoremstyle{plain}
\newtheorem{thm}{Theorem}[section]
\newtheorem*{thm*}{Theorem}
\newtheorem*{cor*}{Corollary}

\newtheorem{prop}[thm]{Proposition}
\newtheorem{lem}[thm]{Lemma}
\newtheorem{cor}[thm]{Corollary}

\newtheorem*{claim*}{Claim}

\theoremstyle{definition}
\newtheorem{defn}[thm]{Definition}
\newtheorem{ex}[thm]{Example}
\newtheorem{rem}[thm]{Remark}
\newtheorem{fact}[thm]{Fact}

\newtheorem{setting}[thm]{Setting}

\theoremstyle{remark}

\numberwithin{equation}{thm}

\def\pd{\operatorname{pd}}

\def\Ker{\operatorname{Ker}}

\def\Coker{\mathrm{Coker}}

\def\e{\mathrm{e}}
\def\m{\mathfrak m}
\def\n{\mathfrak n}

\def\K{\mathrm{K}}

\newcommand{\rma}{\mathrm{a}}

\newcommand{\rme}{\mathrm{e}}

\newcommand{\rmr}{\mathrm{r}}

\newcommand{\rmK}{\mathrm{K}}

\newcommand{\calC}{\mathcal{C}}

\newcommand{\calR}{\mathcal{R}}

\newcommand{\fka}{\mathfrak{a}}
\newcommand{\fkb}{\mathfrak{b}}

\newcommand{\fkm}{\mathfrak{m}}

\newcommand{\fkM}{\mathfrak{M}}
\newcommand{\fkN}{\mathfrak{N}}

\newcommand{\mapright}[1]{%
\smash{\mathop{%
\hbox to 1cm{\rightarrowfill}}\limits^{#1}}}

\newcommand{\mapleft}[1]{%
\smash{\mathop{%
\hbox to 1cm{\leftarrowfill}}\limits_{#1}}}

\begin{document}

\setlength{\baselineskip}{15pt}
\title{The almost Gorenstein Rees algebras of parameters}
\author{Shiro Goto}
\address{Department of Mathematics, School of Science and Technology, Meiji University, 1-1-1 Higashi-mita, Tama-ku, Kawasaki 214-8571, Japan}
\email{goto@math.meiji.ac.jp}
\author{Naoyuki Matsuoka}
\address{Department of Mathematics, School of Science and Technology, Meiji University, 1-1-1 Higashi-mita, Tama-ku, Kawasaki 214-8571, Japan}
\email{naomatsu@meiji.ac.jp}
\author{Naoki Taniguchi}
\address{Department of Mathematics, School of Science and Technology, Meiji University, 1-1-1 Higashi-mita, Tama-ku, Kawasaki 214-8571, Japan}
\email{taniguti@math.meiji.ac.jp}
\author{Ken-ichi Yoshida}
\address{Department of Mathematics, College of Humanities and Sciences, Nihon University, 3-25-40 Sakurajosui, Setagaya-Ku, Tokyo 156-8550, Japan}
\email{yoshida@math.chs.nihon-u.ac.jp}

\thanks{The first author was partially supported by JSPS Grant-in-Aid for Scientific Research 25400051. The second author was partially supported by JSPS Grant-in-Aid for Scientific Research 26400054. The third author was partially supported by Grant-in-Aid for JSPS Fellows 26-126 and by JSPS Research Fellow. The fourth author was partially supported by JSPS Grant-in-Aid for Scientific Research 25400050.}

\begin{abstract} 
There is given a characterization for the Rees algebras of parameters in a Gorenstein local ring to be almost Gorenstein graded rings.  A characterization is also given for the Rees algebras of socle ideals of parameters. The latter one  shows almost Gorenstein Rees algebras rather rarely exist for socle ideals, if the dimension of the base local ring is greater than two.
\end{abstract}

\maketitle

\section{Introduction}\label{intro}
This paper purposes to study the question of when the Rees algebras of given ideals are almost Gorenstein rings. Almost Gorenstein rings are newcomers, which form  a class of Cohen-Macaulay rings that are not necessarily Gorenstein but still good, hopefully next to the Gorenstein rings. The notion of this kind of {\it local} rings dates back to the article \cite{BF} of V. Barucci and R. Fr\"oberg in 1997. They introduced almost Gorenstein rings in the case where the local rings are of dimension one and analytically unramified. One can refer to  \cite{BF} for a beautiful theory of almost symmetric numerical semigroups. Nevertheless, since the notion given by \cite{BF} was not flexible for the analysis of analytically ramified case, in 2013 S. Goto, N. Matsuoka and T. T. Phuong \cite{GMP} extended the notion over arbitrary (but still of dimension one) Cohen-Macaulay local rings. The reader may consult \cite{GMP} for concrete examples of analytically ramified almost Gorenstein local rings as well as generalizations/repairs of results given in \cite{BF}. It was 2015 when S. Goto, R. Takahashi and N. Taniguchi \cite{GTT} finally gave the definition of almost Gorenstein {\it graded/local} rings of higher dimension. We recall here the precise definitions which we need throughout this paper.

\begin{defn}\label{1.1}
Let $(R,\m)$ be a Cohen-Macaulay local ring possessing the  canonical module ${\rm K}_R$. Then we say that $R$ is {\it an almost Gorenstein local ring}, if there exists an exact sequence 
$$
0 \to R \to {\rm K}_R \to C \to 0
$$
of $R$-modules such that  $\mu_R(C) = {\rm e}^0_\mathfrak{m} (C)$, where $\mu_R(C)$ (resp. ${\rm e}^0_\mathfrak{m} (C)$) stands for the number of elements in a minimal system of generators for $C$ (resp. the multiplicity of $C$ with respect to $\mathfrak{m}$).
\end{defn}

\begin{defn}\label{1.2}
Let $R = \bigoplus_{n \ge 0}R_n$ be a Cohen-Macaulay graded ring with $R_0$ a local ring. Suppose that $R$ possesses the graded canonical module $\K_R$. Then $R$ is called {\it an almost Gorenstein graded ring}, if there exists an exact sequence
$$0 \to R \to \mathrm{K}_R(-a) \to C \to 0$$
of graded $R$-modules such that $\mu_R(C) = \e_\fkM^0(C)$, where $\fkM$ is the unique graded maximal ideal of $R$ and $a = \rma (R)$ denotes the $a$-invariant of $R$. Remember that $\mathrm{K}_R(-a)$ stands for  the graded $R$-module whose underlying $R$-module is the same as that of $\K_R$ and whose grading is given by $[\mathrm{K}_R(-a)]_n = [\mathrm{K}_R]_{n-a}$ for all $n \in \Bbb Z$.
\end{defn}

Definition \ref{1.2} means that if $R$ is an almost Gorenstein graded ring, then even though  $R$ is not a Gorenstein ring, $R$ can be  embedded into the graded $R$-module $\mathrm{K}_R(-a)$, so that the difference $\mathrm{K}_R(-a)/R$ is a graded Ulrich $R$-module (see \cite{BHU}, \cite[Section 2]{GTT}) and behaves well. The reader may consult \cite{GTT} about a basic theory of almost Gorenstein graded/local rings and the relation between the graded theory and the local theory. For instance, it is shown  in \cite{GTT} that certain Cohen-Macaulay local rings of finite Cohen-Macaulay representation type, including two-dimensional rational singularities, are almost Gorenstein local rings. The almost Gorenstein local rings which are not Gorenstein are G-regular (\cite[Corollary 4.5]{GTT}) in the sense of \cite{T} and they are now getting revealed to enjoy good properties. However, in order to develop a more theory,  it is still required to find more examples of almost Gorenstein graded/local rings. This observation has strongly motivated the present research.

On the other hand, as for the Rees algebras we nowadays have a satisfactorily developed theory about the Cohen-Macaulay property (see, e.g., \cite{GS, HU, MU, SUV}). Among them Gorenstein Rees algebras are rather rare (\cite{I}).  Nevertheless, as is shown in \cite{GTY1}, some of the non-Gorenstein Cohen-Macaulay Rees algebras can be almost Gorenstein graded rings, which we are eager to report also in this paper.

Let us now state our results, explaining how this paper is organized. Throughout this paper let $(R,\m)$ be a Gorenstein local ring with $d=\dim R$. For each ideal $I$ in $R$ let $\calR (I) = R[It]$~($t$ denotes an indeterminate over $R$)  be the Rees algebra of $I$. We set $\calR = \calR (I)$ and $\fkM = \m\calR + \calR_+$. We are mainly interested in the almost Gorenstein property of $\calR$ and $\calR_\fkM$ in the following two cases. The first one is the case where $I=Q$ is generated by a part $a_1, a_2, \ldots, a_r$ of a system of parameters for $R$. The second one is the case where  $I = Q:\m$, that is  $I$  is the socle ideal of  a full parameter ideal $Q$ of $R$. In Section 2 we  study the first case. We will show that $\calR_\fkM$ is an almost Gorenstein local ring if and only if $R$ is a regular local ring, provided $Q=(a_1, a_2, \ldots, a_r)$ with $r = \mu_R(Q) \ge 3$ (Theorem \ref{3.2}). The result on the almost Gorensteinness in the ring $\calR$ is stated as follows, which is a generalization of \cite[Theorem 8.3]{GTT}.

\begin{thm}[Theorem \ref{3.3}]
Let $R$ be a Gorenstein local ring. Let $a_1, a_2, \ldots, a_r~(r \ge 3)$ be a subsystem of parameters for $R$ and set $Q=(a_1, a_2, \ldots, a_r)$. Then the following conditions are equivalent.
\begin{enumerate}[$(1)$]
\item $\calR(Q)$ is an almost Gorenstein graded ring.
\item $R$ is a regular local ring and $a_1, a_2, \ldots, a_r$ form a part of a regular system of parameters for $R$.
\end{enumerate}
\end{thm}

In Section 3 we shall study the second case where $I = Q:\m$ is the socle ideal of a parameter ideal $Q$ in a regular local ring $R$. The reader may consult  \cite{GTY1} for the case where $\dim R = 2$ and in the present paper we focus our attention on the case where $\dim R \ge 3$. Then somewhat surprisingly we have the following.

\begin{thm}[Theorem \ref{4.5}]
Let $(R, \m)$ be a regular local ring with $d=\dim R \ge 3$ and infinite residue class field. Let $Q$ be a parameter ideal of $R$ such that $Q \neq \m$ and set $I = Q: \m$. Then the following conditions are equivalent.
\begin{enumerate}[$(1)$]
\item $\calR(I)$ is an almost Gorenstein graded ring.
\item Either $I = \m$, or $d=3$ and $I = (x) + \m^2$ for some $x\in \m \setminus \m^2$.
\end{enumerate}
\end{thm}

Theorems 1.3 and 1.4 might suggest that when $\dim R \ge 3$, except the case where $I=Q$ the Rees algebras which are almost Gorenstein graded rings are rather rare. We shall continue the quest also in the future to get more evidence.

In what follows, unless otherwise specified, let $R$ stand for a Noetherian local ring with maximal ideal $\fkm$. For each finitely generated $R$-module $M$ let $\mu_R(M)$ (resp. $\ell_R(M)$) denote the number of elements in a minimal system of generators of $M$ (resp. the length of $M$). We denote by $\e_\fkm^0(M)$ the multiplicity of $M$ with respect to $\fkm$. Let $\mathrm{K}_R$ denote the canonical module of $R$.



\section{The case where the ideals are generated by a subsystem of parameters}

Let $(R, \m)$ be a Gorenstein local ring with $d= \dim R \ge 3$ and let $a_1, a_2, \ldots, a_r ~ (r \ge 3)$ be a subsystem of parameters for $R$. We set $Q = (a_1, a_2, \ldots, a_r)$. Let 
$$
\calR = \calR(Q) = R[Qt] \subseteq R[t]
$$
denote the Rees algebra of $Q$ and set $\fkM = \m \calR+\calR_+$, where $t$ is an indeterminate over $R$. Remember that $\rma (\calR) = -1$. In this section we study the almost Gorenstein property of $\calR$ and $\calR_\fkM$. To do this we need some machinery.

Let $S=R[X_1, X_2, \ldots, X_r]$ be the  polynomial ring over $R$. We consider $S$ as a graded ring with $\deg X_i = 1$ for each  $1 \le i \le r$ and set $\fkN = \m S + S_+$.
Let $\Psi : S \longrightarrow \calR$ be the $R$-algebra map defined by $\Psi(X_i) = a_i t$ for  $1 \le i \le r$. We set $$
{\Bbb A} =\begin{pmatrix}
X_1 & X_2 & \cdots & X_r \\
a_1 & a_2 & \cdots & a_r 
\end{pmatrix}.
$$  Then $\Ker \Psi$ is generated by $2 \times 2$ minors of the matrix $\Bbb A$, that is 
$$
\Ker \Psi = \mathbf{I}_2
\begin{pmatrix}
X_1 & X_2 & \cdots & X_r \\
a_1 & a_2 & \cdots & a_r 
\end{pmatrix},
$$
which is a perfect ideal of $S$ with grade $r-1$. Let 
$$
\calC_{\bullet} \ : \  0 \to C_{r-1} \overset{d_{r-1}}{\to} C_{r-2} \to \cdots \to C_1 \to C_0
$$
 be the Eagon-Northcott complex associated with the matrix $\Bbb A$ (\cite{EN}). Since we are strongly interested in the form of the matrix corresponding to the differentiation $C_{r-1} \overset{d_{r-1}}{\to} C_{r-2}$,  let us briefly remind the reader about the construction of the complex.

Now let $L$ be a finitely generated free $S$-module of rank $r$ with basis $\{T_i\}_{1 \le i \le r}$. We denote by $K = \Lambda L$ the exterior algebra of $L$ over $S$ and let $\Bbb K_{\bullet}(X_1, X_2, \ldots, X_r; S)$ (resp. $\Bbb K_{\bullet}(a_1, a_2, \ldots, a_r; S)$) be the Koszul complex of $S$ generated by $X_1, X_2, \ldots, X_r$ (resp. $a_1, a_2, \ldots, a_r$) with differentiations $\partial_1$ (resp. $\partial_2$). Let $U = S[Y_1, Y_2]$ be the polynomial ring with two indeterminates $Y_1, Y_2$ over $S$. We set $C_0 = S$ and $C_n = K_{n+1} \otimes_S U_{n-1}$ for each $1 \le n \le r-1$. Hence $C_n$ is a finitely generated free $S$-module with free basis
$$
\{ T_{i_1}T_{i_2} \cdots T_{i_{n+1}} \otimes Y_1^{\nu_1}Y_2^{\nu_2} \mid 1 \le i_1 < i_2 < \cdots < i_{n+1} \le r,~\nu_1 + \nu_2 =  n-1 \}.$$  We consider $C_n$ to be  a graded $S$-module so that 
$$\deg(T_{i_1}T_{i_2}\ldots T_{i_{n+1}}\otimes Y_1^{\nu_1}Y_2^{\nu_2}) = \nu_1+1.$$
Then the Eagon-Northcott complex
$$
\calC_{\bullet} \ \ : \ \  0 \to C_{r-1} \to C_{r-2} \to \cdots \to C_1 \to C_0 \to 0
$$
associated with $\Bbb A$ is defined to be a complex of graded free $S$-modules with differentiations 
$$
d_n(T_{i_1}T_{i_2} \cdots T_{i_{n+1}} \otimes Y_1^{\nu_1}Y_2^{\nu_2}) = \sum_{j=1, 2 ~\text{and}~ \nu_j > 0} \partial_j(T_{i_1}T_{i_2} \cdots T_{i_{n+1}}) \otimes Y_1^{\nu_1}\cdots Y_j^{{\nu_j}-1} \cdots Y_2^{\nu_2}
$$
for $n \ge 2$ and 
$$
d_1(T_{i_1}T_{i_2}\otimes 1) = \det
\begin{pmatrix}
X_{i_1} & X_{i_2} \\
a_{i_1} & a_{i_2} 
\end{pmatrix}.
$$
Hence $d_1(C_1) = \mathbf{I}_2(\Bbb A) \subseteq S$. The complex $C_\bullet$ is acyclic and gives rise to a graded minimal $S$-free resolution of $\calR$, since  $\mathbf{I}_2(\Bbb A)
$ is perfect of grade $r-1$ and $X_i, a_i \in \fkN = \m S + S_+$ for all $1 \le i \le r$  (cf. \cite{EN}).

Let $\Bbb M$ denote the matrix of the differentiation $C_{r-1} \overset{d_{r-1} }{\longrightarrow} C_{r-2}$ with respect  to the free basis $\{T_1T_2\cdots T_r\otimes Y_1^iY_2^{r-2-i}\}_{0 \le i \le r-2}$ and $\{T_1 \cdots \overset{\vee}{T_j} \cdots T_r \otimes Y_1^k Y_2^{r-3-k}\}_{1 \le j \le r, ~0 \le k \le r-3}$ of $C_{r-1}$ and $C_{r-2}$, respectively. Then a standard computation gives the following.

\bigskip

\begin{prop}\label{2.2}
{\scriptsize
$$
{}^t{\Bbb M} = 
\begin{pmatrix}
a_1 -a_2 \cdots (-1)^{r+1}a_r & 0 &  &  &  &  \\
X_1 -X_2 \cdots (-1)^{r+1}X_r & a_1 -a_2 \cdots (-1)^{r+1}a_r &    &  &  &  \\
   &   & \ddots  & \\
   &   &  &  &  X_1 -X_2 \cdots (-1)^{r+1}X_r  &  a_1 -a_2 \cdots (-1)^{r+1}a_r\\   
   &   &  &  &  0  &  X_1 -X_2 \cdots (-1)^{r+1}X_r   
\end{pmatrix}.
$$}
\end{prop}

\bigskip

\noindent
We take the $S(-r)$-dual of the resolution $\calC_{\bullet}$ to get the following presentation of the graded canonical module $\mathrm{K}_\calR$ of $\calR$, where 
$
\bigoplus_{i=1}^{r-2}S\left(-(i+1)\right)^{\oplus r}$  and $
\bigoplus_{i=1}^{r-1}S(-i)$
consist of column vectors, say $\bigoplus_{i=1}^{r-1}S(-i)={}^t\left[S(-(r-1))\oplus \cdots \oplus S(-2) \oplus S(-1)\right]$ and $\bigoplus_{i=1}^{r-2}S\left(-(i+1)\right)^{\oplus r}= {}^t\left[S\left(-(r-1)\right)^{\oplus r}\oplus \cdots S\left(-3\right)^{\oplus r}\oplus S\left(-2\right)^{\oplus r}\right]$.

\medskip

\begin{cor}\label{fact1}
$$
\bigoplus_{i=1}^{r-2}S\left(-(i+1)\right)^{\oplus r} \overset{{}^t\Bbb M}{\longrightarrow} \bigoplus_{i=1}^{r-1}S(-i)\overset{\varepsilon}{\longrightarrow}\mathrm{K}_\calR \to 0.$$
\end{cor}

\medskip

\noindent
Hence $\mathrm{r}(\calR) = r-1 \ge 2$, where $\mathrm{r}(\calR)$ denotes the Cohen-Macaulay type of $\calR$.

For each graded $S$-module $M$ and $q \in \Bbb Z$  we denote by $M^{(q)}= \sum_{n\in \Bbb Z}M_{nq}$ the Veronesean submodule of $M$ with degree $q$. Remember that  $M^{(q)}$ is a graded $S^{(q)}$-module whose grading is given by $[M^{(q)}]_n = M_{nq}$ for $n \in \Bbb Z$. We then  have the following. This might be known (see, e.g., \cite{GI2}). Let us note a brief proof in our context.

\begin{prop}
$\calR(Q^{r-1})$ is a Gorenstein ring.
\end{prop}

\begin{proof} Notice that $\calR(Q^{r-1})=\calR^{(r-1)}$. Let $\eta = \varepsilon ({\bf f}) \in [\rmK_\calR]_{r-1}$ in the presentation given by Corollary \ref{fact1} where $\mathbf{f}= \left(\begin{smallmatrix}
1\\
0\\
\vdots\\
0\\
\end{smallmatrix}\right) \in \bigoplus_{i=1}^{r-1}S(-i)$, and set $D = \rmK_{\calR}/\calR \eta$. Then $D_0 = (0)$, since $[\mathrm{K}_\calR]_{r-1} = R\eta$ and we get by Proposition  \ref{2.2} the isomorphism 
$$
D/\m D \cong \bigoplus_{i=1}^{r-2}\left[S/\fkN\right] (-i)
$$
of graded $S$-modules, which shows that $\dim_{\calR_{\fkM}}D_{\fkM} \le d$ and that $D^{(r-1)} = (0)$, because $$D^{(r-1)}/\m D^{(r-1)} = \left[D/\m D\right]^{(r-1)} = (0).$$  We now consider the exact sequence
$$(E_{r-1}) \ \ \ \ \ 
\calR \overset{\psi}{\to} \rmK_{\calR}(r-1) \to D \to 0
$$
of graded $\calR$-modules, where $\psi(1) =\eta$. Then the homomorphism $\psi$ is injective by \cite[Lemma 3.1 (1)]{GTT}, so that applying the functor $[~*~]^{(r-1)}$ to sequence ($E_{r-1}$), we get the isomorphism 
$$\calR^{(r-1)} \cong  \left[\mathrm{K}_\calR\right]^{(r-1)}(-1)$$
of graded $\calR^{(r-1)}$-modules. Thus $\calR (Q^{r-1})=\calR^{(r-1)}$ is a Gorenstein ring, because $\left[\mathrm{K}_\calR\right]^{(r-1)} \cong \mathrm{K}_{\calR^{(r-1)}}$ (cf. \cite{GW1}).
\end{proof}

Before going ahead, let us discuss a little bit more about the presentation
$$
\bigoplus_{i=1}^{r-2}S(-(i+1))^{\oplus r} \overset{{}^t\Bbb M}{\longrightarrow} \bigoplus_{i=1}^{r-1}S(-i)\overset{\varepsilon}{\longrightarrow}\mathrm{K}_\calR \to 0$$
in Corollary \ref{fact1} of the graded canonical module $\mathrm{K}_\calR$ of $\calR$. We set $\xi = \varepsilon ({\bf e}) \in [\rmK_\calR]_1$ where $\mathbf{e} = \left(\begin{smallmatrix}
0\\
\vdots\\
0\\
1
\end{smallmatrix}\right) \in \bigoplus_{i=1}^{r-1}S(-i)$, whence $[\mathrm{K}_\calR]_1 = R\xi$. We set $C = \rmK_\calR/\calR\xi$. Hence 
$$
C \cong \Coker \left[\bigoplus_{i=1}^{r-2}S(-(i+1))^{\oplus r} \overset{\Bbb N}{\longrightarrow} \bigoplus_{i=2}^{r-1}S(-i)\right],
$$
 where $\Bbb N$ denotes the matrix obtained from ${}^t \Bbb M$ by deleting the bottom row, so that Proposition  \ref{2.2} gives the following.

\begin{lem}\label{2.5}
\begin{eqnarray*}
C/S_+C &\cong& S/(S_+ + QS) \otimes_S \left[\bigoplus_{i=2}^{r-1}{S(-i)}\right]\\
&\cong & \bigoplus_{i=2}^{r-1}[R/Q](-i)
\end{eqnarray*}
as graded $S$-modules, where $R=S/S_+$ is considered trivially to be a graded $S$-module.
\end{lem}

In particular $\dim_{\calR_{\fkM}}C_{\fkM} \le d$. Therefore by \cite[Lemma 3.1 (1)]{GTT}  $\dim_\calR C = d$ and the homomorphism $\varphi : \calR \to \rmK_{\calR}(1)$ defined by $\varphi (1) = \xi$ is injective, so that  we get the following.

\begin{cor}\label{fact5} The sequence 
$$0 \to \calR \overset{\varphi}{\longrightarrow} \rmK_\calR(1) \to C \to 0$$
of graded $\calR$-modules is exact and $\dim_\calR C = d$.
\end{cor}

We need the following result to prove Theorem \ref{3.2} below.

\begin{prop}\label{3.1a}
Let $\fka$ be an ideal in a Gorenstein local ring $B$ and suppose that $A = B/\fka$ is an almost Gorenstein local ring. If $A$ is not a Gorenstein ring but  $\pd_B A < \infty$, then $B$ is a regular local ring.
\end{prop}

\begin{proof}
Enlarging it if necessary, we may assume the residue class field of $B$ to be infinite. We choose an exact sequence
$$
0 \to A \to \rmK_A \to C \to 0
$$
of $A$-modules so that $C \neq (0)$ and $C$ is an Ulrich $A$-module. Then $\pd_B \rmK_A < \infty$, because $B$ is a Gorenstein ring and $\pd_BA < \infty$. Hence  $\pd_B C < \infty$. We take an $A$-regular sequence $f_1, f_2, \ldots, f_{d-1} \in \n$~($d = \dim A$)  such  that $\n C = (f_1, f_2, \ldots, f_{d-1}) C$ (this choice is possible; see \cite[Proposition 2.2 (2)]{GTT}) and set $\fkb = (f_1, f_2, \ldots, f_{d-1})$. Then by \cite[Proof of Theorem 3.7]{GTT} we get  an exact sequence
$$
0 \to A/{\fkb A} \to \rmK_A/{\fkb \rmK_A} \to C/{\fkb C} \to 0,
$$
whence $B$ is a regular local ring, because  $\pd_B C/{\fkb C} < \infty$ and $C/{\fkb C} ~(\ne (0))$ is a vector space over $B/\n$.
\end{proof}

\begin{thm}\label{3.2} The following conditions are equivalent.
\begin{enumerate}[$(1)$]
\item $\calR_{\fkM}$ is an almost Gorenstein local ring.
\item $R$ is a regular local ring.
\end{enumerate}
\end{thm}

\begin{proof}
$(1) \Rightarrow (2)~~$ This readily follows from Proposition \ref{3.1a}. Remember that $\calR$ is a perfect $S$-module.

$(2) \Rightarrow (1)~~$  We maintain the same notaion as in Lemma \ref{2.5}. Then 
$$
C/{\m C} \cong (S/\fkN)^{\oplus (r-2)}
$$
 by Lemma \ref{2.5}, whence $\fkM C = \m C$. Therefore $C$ is a graded Ulrich $\calR$-module, because $\dim_\calR C = d$ (cf. Corollary  \ref{fact5}) and $\m$ is generated by $d$ elements. Thus the exact sequence
$$
0 \to \calR_{\fkM} \overset{\calR_{\fkM}~\otimes\varphi}{\longrightarrow} [\rmK_{\calR}]_{\fkM} \to C_{\fkM} \to 0
$$
derived from the sequence in Corollary \ref{fact5} guarantees that $\calR_{\fkM}$ is an almost Gorenstein local ring, because $\rmK_{\calR_\fkM} = [\rmK_\calR]_\fkM$. 
\end{proof}

We are now in a position to study the question of when the Rees algebra $\calR(Q)$ is an almost Gorenstein graded ring. Our answer is the following.

\begin{thm}\label{3.3}
The following conditions are equivalent.
\begin{enumerate}[$(1)$]
\item $\calR$ is an almost Gorenstein graded ring.
\item $R$ is a regular local ring and $a_1, a_2, \ldots, a_r$ form a part of a regular system of parameters for $R$.
\end{enumerate}
\end{thm}

\begin{proof}
$(2) \Rightarrow (1)~~$ We maintain the same notation as in Lemma \ref{2.5}. Firstly choose elements $y_1, y_2, \ldots, y_{d-r} \in \m$ so that $\m = Q + \fka$, where $\fka = (y_1, y_2, \ldots, y_{d-r})$. We then have by Lemma \ref{2.5}   
$$
C/(S_+ +\fka S)C \cong \bigoplus_{i=2}^{r-1}[R/{\m}](-i),
$$ 
so that $\fkN{\cdot}\left[(C/(S_+ + \fka S) C\right] = (0)$. Therefore $C$ is a graded  Ulrich $\calR$-module, whence $\calR$ is an almost Gorenstein graded ring by Corollary \ref{fact5}.

$(1) \Rightarrow (2)~~$ Suppose that $\calR$ is an almost Gorenstein graded ring and consider the exact sequence
$$
0 \longrightarrow \calR \overset{\phi}{\longrightarrow} \rmK_{\calR}(1) \longrightarrow C \longrightarrow 0
$$
of graded $\calR$-modules such that $\mu_\calR(C) = \e^0_{\fkM}(C)$. We set $\rho = \phi (1)$. Then since $\rmr(\calR) = r-1 \ge 2$, we have $\rho = \phi (1) \not\in \m{\cdot}[\mathrm{K}_\calR]_1$ by \cite[Corporally 3.10]{GTT}. Hence $[\mathrm{K}_\calR]_1 =R\rho$ (remember that $[\mathrm{K}_\calR]_1 \cong R$; see Corollary \ref{fact1}). Thus $C\neq (0)$,  $\dim_\calR C=d$,  and $C_n=(0)$ for every $n \le 1$. Therefore $C = \sum_{i=2}^{r-1}S \xi_i$ with $\xi_i \in C_i$ by Corollary \ref{fact1} and hence $Q^{r-2}C =(0)$, because $Q(C/S_+C) = (0)$ by Lemma \ref{2.5}. We set $\fka = (0):_S C$ and $\fkb = \fka \cap \calR$. Hence $Q^{r-2} \subseteq \fkb \subseteq Q$ (see Proposition \ref{2.5}).

\begin{claim*}\label{claim}
$\e^0_{\fkM}(C) = (r-2){\cdot} \e^0_{\m/Q}(R/Q)$.
\end{claim*}

\begin{proof}[{Proof of Claim}] We may assume the field $R/\m$ to be infinite.
We set $\overline{S} = S/\fka$, $A = [\overline{S}]_0 ~(=R/\fkb)$, and $\n$ the maximal ideal of $A$. Notice that $\dim A = d-r$, since $Q^{r-2} \subseteq \fkb \subseteq Q$. Let $B = A[z_1, z_2, \ldots, z_r]$ be the standard graded polynomial ring and let $\psi : B \to \overline{S}$ be the $A$-algebra map defined by $\psi(z_i) = \overline{X_i}$ for each $1 \le i \le r$, where $\overline{X_i}$ denotes the image of $X_i$ in $\overline{S}$. We regard $C$ to be a graded $B$-module via $\psi$. Notice that $\dim_BC = \dim B = d$. Let us choose elements $y_1, y_2, \ldots, y_{d-r}$ of $\m$ so that their images $\{\overline{y_i}\}_{1 \le i \le d-r}$ in $A = R/\fkb$ generate a reduction of $\n$. Then $(\overline{y_i} \mid 1 \le i \le d-r)B + B_+$ is a reduction of the unique graded maximal ideal $\n B + B_+$ of $B$, while the images of $\{y_i\}_{1 \le i \le d-r}$ in $R/Q$ generate a reduction of the maximal ideal  $\m/Q$ of $R/Q$, since $R/Q$ is a homomorphic image of $A=R/\fkb$. Hence setting $\fkN_B = \n B + B_+$, we get
\begin{eqnarray*}
\e^0_{\fkM}(C) 
&=& \e^0_{\fkN_B}(C) \\ 
&=& \ell_{B}(C/\left[(\overline{y_i} \mid 1 \le i \le d-r)B + B_+ \right]C) \\
&=& \ell_{S}(C/\left[(y_i \mid 1 \le i \le d-r)S + S_+ \right]C)\\
&=&(r-2){\cdot}\ell_R(R/[Q + (y_i \mid 1 \le i \le d-r)])\ \ \ (\text{by~Lemma~ \ref{2.5}}) \\
&=& (r-2){\cdot}\e^0_{\fkm/Q}(R/Q) 
\end{eqnarray*}
as claimed.
\end{proof}

Since $\calR$ is an almost Gorenstein graded ring with $\mathrm{r}(\calR) = r-1\ge 2$, we have $\e^0_{\fkM}(C) = r-2$ by \cite[Corollary  3.10]{GTT}, so that $\e^0_{\m/Q}(R/Q)=1$ by the above claim. Thus  $R$ is a regular local ring and $a_1, a_2, \ldots, a_r$ form a part of a regular system of parameters for $R$. 
\end{proof}

\begin{rem}
Let $R=\bigoplus_{n \ge 0}R_n$ be a Cohen-Macaulay graded ring such that $R_0$ is a local ring. Assume that $R$ possesses the graded canonical module $\rmK_R$ and let $\fkM$ denote the graded maximal ideal of $R$. Then because $\rmK_{R_\fkM} = [\rmK_R]_\fkM$, $R_\fkM$ is by definition an almost Gorenstein local ring, once $R$ is an almost Gorenstein graded ring. Theorems \ref{3.2} and \ref{3.3} show that the converse is not true in general. This phenomenon is already recognized by \cite[Example 8.8]{GTT}. See \cite[Section  11]{GTT} for the interplay between the graded theory and the local theory.

\end{rem}

Before closing this section, let us discuss a bit about the case where $r = 2$.

\begin{prop}
Let $(R,\m)$ be a Cohen-Macaulay local ring  and let $a, b$ be a subsystem of parameters for $R$. We set $Q = (a, b)$, $\calR = \calR(Q)$, and $\fkM = \m \calR + \calR_+$. If $\calR_\fkM$ is an almost Gorenstein local ring, then $R$ is a Gorenstein ring, so that  $\calR$ is a Gorenstein ring.
\end{prop}

\begin{proof}
Let $S=R[x, y]$ be the polynomial ring over $R$ and consider the $R$-algebra map  $\Psi : S \to \calR$  defined by  $\Psi(x) = at$, $\Psi(y) = bt$. Then $\Ker \Psi = (bx - ay)$ and $bx - ay \in \fkN^2$, where $\fkN = \m S + S_+$. Therefore since $\calR_\fkM=S_{\fkN}/(bx - ay)S_{\fkN}$ is an almost Gorenstein local ring, by \cite[Theorem 3.7 (1)]{GTT} $S_{\fkN}$ must be a Gorenstein local ring, whence so is $R$. 
\end{proof}

\begin{rem} {\rm Let $(R,\m)$ be a Cohen-Macaulay local ring and let $Q$ be an ideal of $R$ generated by a subsystem $a_1, a_2,\ldots, a_r$ of parameters for $R$. We set $\calR = \calR (Q)$ and $\fkM = \m \calR+\calR_+$. With this setting the authors do not know whether $R$ is necessarily a Gorenstein ring and hence a regular local ring, if $\calR$ (resp. $\calR_\fkM$) is an almost Gorenstein graded (resp. local) ring, provided $r \ge 3$}.
\end{rem}


\section{The case where the ideals are socle ideals of parameters}

In this section we explore the question of when the Rees algebras of socle ideals are almost Gorenstein. In what follows, let $(R, \m)$ be a Gorenstein local ring of dimension $d \ge 3$ with infinite residue class field. Let $I$ be an $\m$-primary ideal of $R$. We assume that our ideal $I$ contains a parameter ideal $Q= (a_1, a_2, \ldots, a_d)$ of $R$ such that $I^2=QI$. We set $J = Q:I$, 
$
\calR = R[It]\subseteq R[t]
$
($t$ an indeterminate over $R$), and $\fkM = \m \calR + \calR_+$. Notice that $\calR$ is a Cohen-Macaulay ring (\cite{GS}) and $\rma (\calR) = -1$. We are interested in the question of when $\calR$ (resp. $\calR_\fkM$) is an almost Gorenstein graded (resp. local) ring.

Let us  note the following.

\begin{thm}[{\cite[Theorem 2.7]{U}}]\label{3.1}
$$
\rmK_{\calR}(1) \cong \sum_{i=0}^{d-3}\calR{\cdot} t^i + \calR{\cdot}J t^{d-2}  
$$
as a graded $\calR$-module. 
\end{thm}

 As a direct consequence we get the following.

\begin{cor}\label{fact3}
$\rmr(\calR) = (d-2) + \mu_R(J/I)$.	
\end{cor}

Here  $\rmr(\calR)$ denotes  the Cohen-Macaulay type of $\calR$. Consequently, $\calR$ is a Gorenstein ring if and only if $d=3$ and $I=J$, that is $I$ is a good ideal in the sense of \cite{GIW}.

Let us begin with the following.

\begin{lem}\label{4.2}
Suppose that $Q \subseteq \m^2$. Then $\mu_R(\m Q) = d{\cdot} \mu_R(\m)$.
\end{lem}

\begin{proof}
Let $\sigma : \m \otimes_R Q \to \m Q$ be the $R$-linear map defined by $\sigma(x \otimes y) = xy$ for all $x \in \m$ and $y \in Q$. To see $\mu_R(\m Q) = d{\cdot} \mu_R(\m)$, it is enough to show that $$\Ker \sigma \subseteq \m{\cdot}\left[\m \otimes_R Q\right].$$ Let $x \in \Ker \sigma$ and write $x= \sum_{i=1}^d f_i \otimes a_i$ with $f_i \in \m$. Then since $\sum_{i=1}^d a_if_i  = 0$ and $a_1, a_2, \ldots, a_d$ form an $R$-regular sequence, for each $1 \le i \le d$ we have $f_i \in (a_1, \ldots, \overset{\vee}{a_i}, \ldots, a_d) \subseteq \m^2$. Hence $x \in \m{\cdot}\left[\m \otimes_R Q\right]$ as required.
\end{proof}

\begin{thm}\label{4.1}
If  $J=\m$ and $I \subseteq \m^2$, then $\calR_{\fkM}$ is not an almost Gorenstein local ring.  
\end{thm}

\begin{proof}[Proof of Theorem \ref{4.1}.]
We set $A = \calR_\fkM$ and suppose that $A$ is an almost Gorenstein local ring. Notice that $A$ is not a Gorenstein ring, since $J \ne I$ (Corollary \ref{fact3}). We choose an exact sequence
$$
0 \to A \overset{\varphi}{\to} \mathrm{K}_A \to C \to 0
$$
of $A$-modules with $C\neq (0)$ and $C$ an Ulrich $A$-module. Let $\n$ denote the maximal ideal of $A$ and choose elements $f_1, f_2, \ldots, f_d \in \n$ so  that $\n C = (f_1, f_2, \ldots, f_d) C$. Let $\xi =\varphi (1)$. Then because $\xi \not\in \n\rmK_A$ by \cite[Corollary 3.10]{GTT}, we get 
$$
\mu_{A}(\n C) \le d{\cdot}(r-1),
$$
where $r = \rmr(A) = (d-2)+\mu_R(J/I)$ (Corollary \ref{fact3}). As  $\xi \notin \n \rmK_A$, we also have the exact sequence
$$
0 \to \n \xi \to \n \rmK_{A} \to \n C \to 0.
$$
Therefore because $\rmK_A = \left[\mathrm{K}_\calR\right]_\fkM$,  we get the estimation  
\begin{eqnarray*}
\mu_{\calR}(\fkM \rmK_{\calR}) = \mu_{A}(\n \rmK_{A}) &\le & \mu_{A}(\n C) + \mu_{A}(\n) \\
& \le & d{\cdot}\left[(d-2)+\mu_R(J/I)-1\right] + \left[\mu_R(\m) + \mu_R(I)\right].
\end{eqnarray*}
On the other hand, since $\fkM = (\m, It)\calR$ and $\rmK_\calR (1)= \displaystyle\sum_{i=0}^{d-3}\calR{\cdot} t^i + \calR{\cdot} J t^{d-2}$ by Theorem \ref{3.1}), it is straightforward to check that 
$$
\mu_{\calR}(\fkM \rmK_{\calR}) = (d-2) {\cdot}\mu_R(\m) + \mu_R(I + \m J) + \mu_R(IJ/I^2).
$$
Therefore
$$
\left[\mu_R(I + \m J) + \mu_R(IJ/{I^2})\right] -\left[\mu_R(I) + d{\cdot} \mu_R(J/I)\right] \le (d-3){\cdot}\left[d-\mu_R(\m)\right] \le 0,
$$
whence 
$$(*)\ \ \ \ \ 
\mu_R(I + \m J) + \mu_R(IJ/{I^2}) \le \mu_R(I) + d{\cdot} \mu_R(J/I).
$$
We now use the hypothesis that $J = \m$ and $I \subseteq \m^2$. Notice that $\m I = \m Q$, since $I = Q:\m$ and $Q$ is a minimal reduction of $I$. Then by the above estimation ($*$) we get  
$$
\mu_R(\m^2) + \mu_R(\m Q) \le \mu_R(I) +  d{\cdot} \mu_R(\m), 
$$
whence $\mu_R(\m^2) \le \mu_R(I)$ by Lemma \ref{4.2}. Therefore
$$\binom{d+1}{2} \le \mu_R(\m^2) \le \mu_R(I) = d+1,$$ which is impossible, because $d \ge 3$. 
\end{proof}

\begin{cor}\label{corollary3.5}
Let $Q$ be a parameter ideal of $R$ such that $Q \subseteq \m^2$. Then $\calR_{\fkM}$ is not an almost Gorenstein local ring, where $\calR = \calR (Q:\m)$ and $\fkM = \m \calR+\calR_+$.
\end{cor}

\begin{proof} Let $I = Q:\m$. Then $I^2 = QI$ and $I \subseteq \m^2$ by \cite[Theorem 1.1]{W}, while $Q: I = Q : (Q :\m) = \m$, since  $R$ is a Gorenstein ring.\end{proof}

Let us  study the case where $R$ is a regular local ring. The goal is the following.

\begin{thm}\label{4.5}
Let $(R, \m)$ be a regular local ring of dimension $d\ge 3$ with infinite residue class field. Let $Q$ be a parameter ideal of $R$. Assume that $Q \neq \m$ and set $I = Q: \m$. Then the following conditions are equivalent.
\begin{enumerate}[$(1)$]
\item $\calR(I)$ is an almost Gorenstein graded ring.
\item Either $I = \m$, or $d=3$ and $I = (x) + \m^2$ for some $x\in \m \setminus \m^2$.
\end{enumerate}
\end{thm}

We divide the proof of Theorem \ref{4.5} into several steps. Let us begin with the case where $Q \nsubseteq \m^2$. Our setting is the following.

\begin{setting}\label{4.6}
Let $(R, \m)$ be a regular local ring of dimension $d\ge 3$ with infinite residue class field. We write $\m=(x_1, x_2, \ldots, x_d)$. Let $Q$ be a parameter ideal of $R$ and let $1 \le i \le d-2$ be an integer. For the ideal $Q$ and the integer $i$ we assume that $$(x_j \mid 1 \le j \le i) \subseteq  Q \subseteq (x_j \mid 1 \le j \le i) + \m^2.$$
\end{setting}

 We set $\fka = (x_j \mid 1 \le j \le i)$, $\fkb=(x_{j} \mid i+1 \le j \le d)$, and $I = Q :\m$. Hence $Q = \fka + (a_j \mid i+1 \le j \le d)$ with $a_j \in \fkb^2$, so that we have the presentation 
$$
\begin{pmatrix}
x_1 \\
x_2 \\
\vdots \\
x_i \\
a_{i+1} \\
\vdots \\
a_d
\end{pmatrix}
=
\begin{pmatrix}
1 &  &    & & &  & & \\
 & \ddots &  & & & 0 & &\\
  &  & 1 & & & &  & &\\
    &  &  & & & &  & &\\
    & 0 &  & & & \alpha_{jk}&  & &\\
    &  &  & & & &  & & \\
\end{pmatrix}
\begin{pmatrix}
x_1 \\
x_2 \\
\vdots \\
x_i \\
x_{i+1} \\
\vdots \\
x_d
\end{pmatrix}
$$
with $\alpha_{jk} \in \fkb$ for every $i+1 \le j, k \le d$. Let $\Delta = \det (\alpha_{jk})$. Then $\Delta \in \fkb^2$ and $Q : \Delta= \m$ by \cite[Theorem 3.1]{H}, whence $I= Q + (\Delta)$. Consequently we have the following. See Corollary  \ref{fact3} for assertion (4). 

\begin{prop}\label{4.7}
The following assertions hold true.
\begin{enumerate}[$(1)$]
\item $I^2 = QI$.
\item $Q : I = \m$.
\item $I \subseteq \fka  + \fkb^2$. 
\item $\mu_R(\m/I ) = d-i$ and $\rmr(\calR)= 2d -(i + 2)$.
\end{enumerate}
\end{prop}

\begin{prop}\label{4.8}
$\mu_R(\m Q/I^2) = d(d -i)$.
\end{prop}

\begin{proof} Since $\m = \fka + \fkb$ and $\fka \subseteq Q$, we get 
$$
\m Q/{[ I^2 + \m^2 Q]} \cong \fkb Q /[\fkb Q \cap (Q^2 + \m \fkb Q)],
$$ 
while $Q \cap \fkb \subseteq \m \fkb$, since $Q = \fka +(a_j \mid i+1 \le j \le d)$ and $a_j \in \fkb^2$ for every $i+1 \le j \le d$. Hence $Q^2 \cap \fkb Q \subseteq \m \fkb Q$, so that   
$$
\fkb Q \cap (Q^2 + \m \fkb Q) = \m \fkb Q.
$$
Therefore $\mu_R(\m Q/I^2) = \mu_R(\fkb Q)$. The method in the proof of Lemma \ref{4.2} works to get $\mu_R(\fkb Q)= \mu_R(\fkb \otimes_R Q) = d (d -i)$ as claimed.
\end{proof}

The following is the heart of the proof. 

\begin{prop}\label{4.9}
Suppose that $\calR(I)$ is an almost Gorenstein graded ring. Then $d=3$ and $I = (x_1)  + \m^2$. 
\end{prop}

\begin{proof}
Since $\rmr(\calR) = 2d -(i+2) \ge 3$ by Proposition \ref{4.7} (4),  $\calR$ is not a Gorenstein ring. We take an exact sequence
$$
0 \to \calR \overset{\varphi}{\to} \rmK_{\calR}(1) \to C \to 0 
$$
of graded $\calR$-modules so that $C \neq (0)$ and $\mu_{\calR}(C) = \rme^0_{\fkM}(C)$. Since $[\rmK_\calR]_1 \cong R$ (see Corollary \ref{fact3}) and $\xi = \varphi(1) \notin \fkM {\cdot}[\rmK_{\calR}(1)]$, $\xi$ is a unit of $R$. Therefore the isomorphism of Theorem  \ref{3.1} shows 
$$
C \cong \left[\sum_{i=1}^{d-3}\calR{\cdot}t^i + \calR {\cdot}\m t^{d-2}\right]/\calR_+,
$$ 
from which by a direct computation we get  the following. 

\begin{fact}\label{claim2}
$$
\mu_{\calR}(\fkM C) = 
\left\{
\begin{array}{ll}
\mu_R(\m^2/{I\cap \m^2}) + \mu_R(\m Q /I^2) & (d=3),\\
(d-i) + d(d-4) + \mu_R(I + \m^2) + \mu_R(\m Q /I^2) & (d \ge 4).
\end{array}
\right.
$$
\end{fact}
\noindent
On the other hand we have $\mu_{\calR_{\fkM}}~(C_{\fkM}) = \rmr(\calR) -1 = 2d -(i+3)$ by \cite[Corollary 3.10]{GTT}. Consequently
$$
\mu_{\calR}(\fkM C) = \mu_{\calR_{\fkM}}~(\fkM C_{\fkM}) \le d{\cdot}\left(2d - (i+3)\right),
$$
because $C_{\fkM}$ is an Ulrich $\calR_\fkM$-module with $\dim_{\calR_{\fkM}}C_{\fkM} = d$ (\cite[Proposition 2.2 (2)]{GTT}). Assume now that $d \ge 4$. Then by Fact \ref{claim2} and Proposition \ref{4.8} we have 
$$
(d-i) +  d(d-4) + \mu_R(I + \m^2) + d (d-i) \le d (2d - (i+3)),
$$
so that $\mu_R(I + \m^2) \le i$. This is impossible, because 
$$
\mu_R(I + \m^2) = \mu_R(\fka + \m^2) = i + \binom{d-i+1}{2} > i.
$$
Therefore we get $d = 3$ and $i=1$, whence 
$$
\mu_R(\m^2/[I\cap \m^2]) + \mu_R(\m Q /I^2) \le 6, 
$$
so that we have $\m^2 = {I\cap \m^2} \subseteq I$, because $ \mu_R(\m Q /I^2) = 6$ by Proposition \ref{4.8}. Thus $I = (x_1) + \m^2$ by Proposition  \ref{4.7} (3).
\end{proof}

We are now in a position to finish the proof of Theorem \ref{4.5}.

\begin{proof}[Proof of Theorem \ref{4.5}]
$(1) \Rightarrow (2)$~~If $Q$ is integrally closed in $R$, then by \cite[Theorem 3.1]{G} $Q = (x_1, x_2, \ldots, x_{d-1}, x_d^q)$ for some regular system $\{x_i\}_{1 \le i \le d}$ of parameters of $R$ and for some integer $q \ge 1$. Therefore 
$$
I = Q: \m = Q : x_d = (x_1, x_2, \ldots, x_{d-1}, x_d^{q-1}),
$$
so that we have $q=2$ by Theorem \ref{3.3}, that is $I = \m$. Suppose that $Q$ is not integrally closed in $R$. Then  $Q \nsubseteq \m^2$ by Corollary \ref{corollary3.5}. Let $\{x_j\}_{1 \le j \le d}$ be a regular system of parameters for  $R$ and take  the integer $1 \le i \le d-2$ so that  $(x_j \mid 1 \le j \le i) \subseteq  Q \subseteq (x_j \mid 1 \le j \le i) + \m^2$  (cf. \cite[Theorem 3.1]{G}). We then have $d = 3$ and $I =  (x_1) + \m^2$ by Proposition \ref{4.9}.

$(2) \Rightarrow (1)$~~This follows from Theorem \ref{3.3} and the following proposition. 
\end{proof}

\begin{prop}\label{4.10}
Let $(R, \m)$ be a Gorenstein local ring with $\dim R=3$ and infinite residue class field. Let $Q$ be a parameter ideal of $R$. Assume that $Q \neq \m$  and set $I = Q: \m$. If $I^2 = QI$ and $\m^2 \subseteq I$, then $\calR(I)$ is an almost Gorenstein graded ring.
\end{prop}

\begin{proof}
We have $\rmK_{\calR}(1) = \calR + \calR{\cdot}\m t$. Consider the exact sequence
$$
0 \to \calR \overset{\varphi}{\to} \rmK_{\calR}(1) \to C \to 0
$$
of graded $\calR$-modules with $\varphi(1) = 1$. Then since $\m^2 \subseteq I$ and $\m I = \m Q$, we readily see
$$
\fkM \left[\calR{\cdot} \m t\right]  \subseteq \calR + Qt\left[\calR{\cdot} \m t\right] 
$$
which shows $C$ is a graded  Ulrich $\calR$-module (see \cite[Proposition 2.2 (2)]{GTT}. Thus $\calR$ is an almost Gorenstein graded ring.
\end{proof}

Let us note one example.

\begin{ex}
Let $R = k[[x, y, z]]$ be the formal power series ring over an infinite field $k$.  We set $\m = (x, y, z)$, $Q = (x, y^2, z^n)$ with $n \ge 2$,  and $I =Q: \m$. Then $I= (x, y^2, yz^{n-1}, z^n)$ and $I^2 =QI$. 
\begin{enumerate}[$(1)$]
\item If $n = 2$, then $I = (x) + \m^2$, so that $\calR(I)$ is an almost Gorenstein graded ring.
\item Suppose $n \ge 3$. Then $I \neq \m$, $Q \neq \m$, and $I \neq (f) + \m^2$ for any $f \in \m \setminus \m^2$. Hence $\calR(I)$ is not an almost Gorenstein graded ring.
\end{enumerate}
\end{ex}





\end{document}